\documentclass{article}
\usepackage{titlesec}
\usepackage[margin = 0.96in]{geometry}
\usepackage{amsmath, amssymb, mathtools, amsfonts, amsthm, bbm}
\usepackage{xcolor}
\newtheorem{thm}{Theorem}[section]
\newtheorem{lem}{Lemma}[section]
\newtheorem{prop}{Proposition}[section]
\usepackage{hyperref}

\titleformat*{\section}{\large\bfseries}

\hypersetup{
	colorlinks=true,
	linkcolor=blue,
	filecolor=magenta,      
	urlcolor=cyan,
}

\renewcommand{\P}{\Phi}
\newcommand{\Z}{\mathbb{Z}}
\renewcommand{\l}{\langle}
\renewcommand{\r}{\rangle}
\newcommand{\GL}{\mathrm{GL}}
\newcommand{\Ind}{\mathrm{Ind}}
\newcommand{\ord}{\mathrm{ord}}

\newcommand{\Res}{\mathrm{Res}}

\title{Non-Induced Representations of Finite Cyclic Groups}
\author{Ramanujan Srihari\thanks{\textsc{Department of Mathematical Sciences, Indian Institute of Science Education and Research, Mohali, India}. \emph{Email address}: \texttt{ramanujan.srihari@gmail.com}.}}
\date{}

\begin{document}

\maketitle

\begin{abstract}
	Let $K$ be an algebraically closed field of characteristic $0$ and let $G$ be a finite cyclic group of order $n$. In this note we prove, using induction on the number of prime divisors of $n$, that $R_K(G)/ I \cong \Z[X]/\l\Phi_n(X)\r$ where $R_K(G)$ denotes the ring of $K$-representations of $G$ and $I$ is the sum of ideals $\Ind_H^G(R_K(H))$ of $R_K(G)$ as $H$ varies over all proper subgroups of $G$. This gives us an idea of how many representations of $G$ are \emph{not} induced from representations of a proper subgroup.
\end{abstract}

\section{Introduction}

Let $K$ be a field and let $G$ be an Abelian group. By a $K$-\emph{representation} of $G$ we mean a pair ($\rho$, $V$) where $V$ is a $K$-vector space and $\rho: G \to \GL(V)$ is a group homomorphism. Moreover, $(\rho, V)$ is said to be \emph{finite dimensional} if $V$ is a finite dimensional $K$-vector space. Two representations $(\rho_1, V_1)$ and $(\rho_2,V_2)$ are said to be isomorphic if there exists a $K$-linear isomorphism $\varphi: V_1 \to V_2$ such that $\varphi\circ\rho_1 = \rho_2\circ\varphi$. (In fact if $\dim_KV = n$ and if we choose a basis for $V$, then saying that the two representations $\rho_1$ and $\rho_2$ are isomorphic is the same as saying that the matrices corresponding to $\rho_1$ and $\rho_2$ are conjugate over $\GL_n(K)$.) It is easily verified that the relation of isomorphism on the set of all $K$-representations of $G$ defines an equivalence relation. For a representation $(\rho, V)$ of $G$ we let $[(\rho, V)]$ denote the isomorphism class of $(\rho, V)$. 

\

Let $F_K(G) = \bigoplus_{(\rho, V)} \Z[(\rho, V)]$ be the free Abelian group generated by the isomorphism classes of finite dimensional irreducible $K$-representations of $G$, and let $S_K(G) \subseteq F_K(G)$ be the subgroup generated by the elements $[(\psi_1, V_1)] + [(\psi_2,V_2)] - [(\psi_1\oplus\psi_2, V_1\oplus V_2)]$ as $\psi_1$ and $\psi_2$ vary over all finite dimensional irreducible $K$-representations of $G$. Then the \emph{representation ring} $R_K(G)$ of $G$ is defined by
\begin{equation*}
	R_K(G) := \frac{F_K(G)}{S_K(G)}.
\end{equation*}
Although \emph{a priori} $R_K(G)$ is just an Abelian group, we can equip it with a product `$\times$': for two isomorphism classes $[(\rho_1, V_1)]$ and $[(\rho_2, V_2)]$, define $[(\rho_1, V_1)]\times[(\rho_2, V_2)] = [(\rho_1\otimes\rho_2, V_1\otimes V_2)]$, and extend this product $\Z$-linearly to $R_K(G)$. (Here and in the sequel, $\otimes$ denotes tensor product over $K$.) That this product is well defined in $R_K(G)$ follows from the fact that tensor product distributes over direct sum. From the same reason it also follows that equipped with this product, $R_K(G)$ forms a commutative ring with identity $[(\mathbbm{1}_G, K)]$, the class of the trivial representation of $G$.

\

We recall that given a $K$-representation $(\rho, V)$ of $G$, we may view $V$ as a module over the commutative algebra $K[G]$. Indeed, any $g \in G$ acts on a vector $v \in V$ as $v \mapsto \rho(g)(v)$. Conversely, given a $K[G]$-module $V$ we obtain a representation $\rho: G \to \GL(V)$ by defining $\rho(g)(v) = g\cdot v$. Thus in what follows, we interchangeably use $\rho$ or $V$ to denote the representation $(\rho, V)$, where we view $V$ as a $K[G]$-module via the action of $\rho$.

\

Given a subgroup $H$ of $G$, we view $K[H]$ as a $K$-subalgebra of $K[G]$ via the natural inclusion $H \hookrightarrow G$. Thus given a representation $V$ of $G$, we may view $V$ as a $K[H]$-module, and therefore as a representation of $H$. We denote this representation of $H$ by $\Res^G_H(V)$. Conversely to a representation $W$ of $H$ we may associate a representation of $G$ by defining
\begin{equation*}
	\Ind_H^G(W) := W\otimes_{K[H]} K[G].
\end{equation*}
In other words, $\Ind_H^G(W)$ is obtained by base change of the $K[H]$-module $W$ to $K[G]$. Since $\Ind_H^G(W_1\oplus W_2) \cong \Ind_H^G(W_1)\oplus \Ind_H^G(W_2)$, we obtain a well-defined map at the level of representation rings (which, by abuse of notation, we also denote by $\Ind_H^G$) by extending $\Ind_H^G$ $\Z$-linearly, $\Ind_H^G: R_K(H) \to R_K(G)$. It is easily seen that $\Ind_H^G$ is a homomorphism of Abelian groups. Furthermore, for any representations $V$ of $G$ and $W$ of $H$, we have the following $K[G]$-module isomorphisms.
\begin{eqnarray}
V\otimes \Ind_H^G(W) & = & V\otimes (W\otimes_{K[H]} K[G])\nonumber \\
& \cong & (V\otimes W)\otimes_{K[H]} K[G] \nonumber \\
& \cong & (\Res^G_H(V)\otimes W)\otimes_{K[H]} K[G] \nonumber \\
& = & \Ind_H^G(\Res^G_H(V)\otimes W). \nonumber 
\end{eqnarray}
This shows that the image $\Ind_H^G(R_K(H))$ is in fact an ideal in $R_K(G)$. Thus we may form the quotient ring $R_K(G)/\Ind_H^G(R_K(H))$, which can be thought of as a measure of the representations of $G$ which are \emph{not} induced from $H$. 

\

In this note we wish to get a sense of how many representations of a finite cyclic group are not induced from representations of a proper subgroup. In other words, we wish to understand the ring
\begin{equation*}
	\frac{R_K(G)}{\sum_{H < G} \Ind_H^G(R_K(H))}
\end{equation*}
where $G$ is a finite cylic group and as $H$ ranges over all proper subgroups of $G$. To that end we have the following result.

\begin{thm}\label{mainthm}
Let $G$ be a finite cyclic group of order $n$, and let $K$ be an algebraically closed field of characteristic $0$. Then 
\begin{equation}
\frac{R_K(G)}{\sum_{H < G} \Ind_H^G(R_K(H))} \cong \frac{\Z[X]}{\l\P_n(X)\r}
\end{equation}
where $\P_n$ denotes the $n^\mathrm{th}$ cyclotomic polynomial.
\end{thm}

\

\noindent\textbf{Organisation.} In Section \ref{sec2} we compute the representation ring of a finite cyclic group $G$ in terms of a primitive character $\chi$ of $G$. More specifically, we show that $R_K(G) \cong \Z[\chi]$. We then use this isomorphism in Section \ref{sec3} to compute the sum of ideals $\Ind_H^G(R_K(H))$ as $H$ ranges over all proper subgroups of $G$. This allows us to reformulate Theorem \ref{mainthm} in the form of Theorem \ref{reform}, which is precisely what we prove in Section \ref{sec5} after establishing a few properties of the cyclotomic polynomials in Section \ref{sec4}.

\

In passing, we remark that Theorem \ref{reform} is interesting in its own right. For any positive integer $n$ and a divisor $d$ of $n$, it tells us that the ideal generated by polynomials of the type $P_{n,d}$ (see Section \ref{sec3} for the definition of $P_{n,d}$) in $\Z[X]$ is in fact a principal ideal, generated by $\Phi_n(X)$. To look at it in another way, Theorem \ref{reform} says that the ideal generated by $\Phi_n(X)$ (the coefficients of which are somewhat mysterious) in $\Z[X]$ can be generated by the polynomials $P_{n,d}$, which are just sums of powers of $X$. Readers who wish to see the proof of Theorem \ref{reform} can skip directly to Section \ref{sec4}.

\

\noindent\textbf{Acknowledgements.} I am very grateful to Professor Kapil Hari Paranjape for introducing me to this problem, and for suggesting improvements and corrections to an earlier version of the note. In particular, the reformulation of Theorem \ref{mainthm} as Theorem \ref{reform} is due to him. I also thank Suneel Kumar for directing me to a couple of results in \cite{Ser} (namely, \cite[Cor.~2, Thm.~4, Ch.~2]{Ser} and \cite[Thm.~12, Ch.~3]{Ser}) which were used in proving Proposition \ref{Ind1calc}.

\section{Representation Ring of Finite Cyclic Groups}\label{sec2}

In this section we determine the representation ring of finite cyclic groups. Since any finite dimensional representation can be written as a sum of irreducible representations, it suffices to classify the irreducible representations.



\begin{lem}\label{lem2}
Let $K$ be an algebraically closed field and $G$ an Abelian group. Let $\rho: G \to \GL(V)$ be a finite dimensional irreducible $K$-representation of $G$. Then for every $g \in G$, $\rho(g)$ acts on $V$ as $\lambda_g\cdot\mathbbm{1}_V$ for some $\lambda_g \in K^\times$. 
\end{lem}

\begin{proof}
Fix $g \in G$. Since $K$ is assumed to be algebraically closed, the characteristic polynomial of $\rho(g)$ has a root, say, $\lambda_g \in K^\times$, so that $\{0\} \subsetneq \ker(\rho(g)-\lambda_g\cdot\mathbbm{1}_V)$. Moreover, as $G$ is assumed to be Abelian, $\ker(\rho(g)-\lambda_g\cdot\mathbbm{1}_V)$ is in fact a non-zero $K[G]$-submodule of $V$. However, $\rho$ is irreducible which implies that $V$ is a simple $K[G]$-module. Therefore $\ker(\rho(g)-\lambda_g\cdot\mathbbm{1}_V) = V$, or $\rho(g) = \lambda_g\cdot\mathbbm{1}_V$.
\end{proof}

Lemma \ref{lem2} says that if $(\rho, V)$ is a finite dimensional irreducible representation of $G$, then $\rho(g)$ acts as a scalar on $V$ for every $g \in G$, so that any subspace of $V$ will be stable under the action of $G$. Since $\rho$ is assumed to be irreducible, we conclude that $\dim_KV = 1$. Thus we have proved the following lemma.

\begin{lem}\label{lem3}
	Let $K$ be an algebraically closed field and $G$ an Abelian group. Then any irreducible $K$-representation of $G$ is one-dimensional.
\end{lem}

The results in this section so far hold for any irreducible $K$-representation of an Abelian group $G$, as long as $K$ is algebraically closed. We now specialize to finite cyclic groups. Let $G$ be a finite cyclic group of order $n$ with a generator $\sigma$. Then by Lemma \ref{lem3}, any irreducible representation of $G$ is a character $\chi: G \to K^\times$. Fix a primitive $n^\mathrm{th}$ root of unity $\zeta_n$ in $K$. Then the representation $\chi_G$ defined on the generator $\sigma$ of $G$ by $\chi_G(\sigma) = \zeta_n$ is an irreducible representation, since it is one-dimensional. We shall refer to $\chi_G$ as `the' primitive character of $G$. The following proposition shows that any other irreducible representation of $G$ is just a power of $\chi_G$.

\begin{prop}\label{prop}
	Any irreducible representation of $G$ is of the form $\chi_G^k$ where $0 \leq k < n$. (Here $\chi_G^k$ denotes the $k$-fold tensor product of $\chi_G$, $\chi_G^{\otimes k}$.)
\end{prop}

\begin{proof}
	Indeed, since the order of $G$ is $n$, we see for any character $\rho$ that $\rho(\sigma)$ must be a $n^\mathrm{th}$ root of unity. Writing $\rho(\sigma) = \zeta_n^k$ for some $0 \leq k < n$, it follows that $\rho = \chi_G^k$.
\end{proof}

The above proposition provides us with a clean expression for the representation ring of $G \cong \Z/n\Z$:
\begin{equation}\label{cyclring}
	R_K(G) \cong \Z[\chi_G] \cong \frac{\Z[X]}{\langle X^n-1\rangle}
\end{equation}
where the second isomorphism is defined by mapping $\chi_G$ to $X$.

\section{Reformulation of Theorem \ref{mainthm}} \label{sec3}

As before, let $G$ be a cyclic group of order $n$ generated by $\sigma$ and let $\chi_G$ denote the primitive character of $G$, i.e., $\chi_G(\sigma) = \zeta_n$. For a positive divisor $d$ of $n$, let $H = \l \sigma^d \r$ be the unique cyclic subgroup of $G$ of order $n/d$. Then the restriction $\chi_G|_H$ is the primitive character $\chi_H$ of $H$. Indeed, $\chi_G(\sigma^d) = \zeta_n^d = \zeta_{n/d}$. It then follows that every irreducible representation of $H$ (which is a character by Lemma \ref{lem3} and therefore a power of $\chi_H$ by Proposition \ref{prop}) is the restriction of a representation of $G$. Since any finite dimensional representation is a sum of irreducible representations, we conclude that any finite dimensional representation of $H$ arises as the restriction of some representation of $G$. 

\

Next we would like to determine the ideal $\Ind_H^G(R_K(H))$. To that end, let $(V,\rho)$ be a representation of $H$. By the discussion in the previous paragraph, $(V,\rho)$ is a restriction of a representation $(W, \psi)$ of $G$: $\rho = \psi|_H$. Let $\mathbbm{1}_H: H \to K^\times$ denote the trivial representation of $H$. Then
\begin{eqnarray}
	\Ind_H^G(V) & = & \Ind_H^G(\Res^G_H(W)) \nonumber \\
	& \cong & W \otimes_{K[H]} K[G] \nonumber \\
	& \cong & (W \otimes K) \otimes_{K[H]} K[G] \nonumber \\
	& \cong & W \otimes (K\otimes_{K[H]} K[G]) \nonumber \\
	& \cong & W \otimes \Ind_H^G(\mathbbm{1}_H) \nonumber
\end{eqnarray}
where in the above $K[G]$-module isomorphisms we view $K$ as a $K[H]$-module where $H$ acts trivially on $K$. (That is, $K$ is the trivial representation of $H$.) This shows that the ideal $\Ind_H^G(R_K(H))$ of $R_K(G)$ is the principal ideal generated by $\Ind_H^G(\mathbbm{1}_H)$. Therefore in order to determine $\Ind_H^G(R_K(H))$, it suffices to determine $\Ind_H^G(\mathbbm{1}_H)$. This is precisely the content of the following proposition.

\begin{prop} \label{Ind1calc}
	With the above notations,
	\begin{equation}
		\Ind_H^G(\mathbbm{1}_H) \cong \bigoplus_{k = 0}^{d-1}\chi_G^{kn/d}.
	\end{equation}
\end{prop}
\begin{proof}
	By \cite[Cor.~2, Thm.~4, Ch.~2]{Ser} it suffices to show that the characters of $\rho:= \Ind_H^G(\mathbbm{1}_H)$ and $\psi:= \bigoplus_{k = 0}^{d-1}\chi^{kn/d}$ agree. First, the character $\chi_\psi$ of $\psi$ is just $\chi_\psi = \sum_{k=0}^{d-1}\chi_G^{kn/d}$. Let $u \in G$. If $u \in H = \l \sigma^d \r$, then $u = \sigma^{d\ell}$ for some positive integer $\ell$. So,
	\begin{equation*}
		\chi_\psi(u) = \chi_\psi(\sigma^{d\ell}) = \sum_{k=0}^{d-1}\chi_G^{kn/d}(\sigma^{d\ell}) = \sum_{k=0}^{d-1}\zeta_n^{knl} = d.
	\end{equation*}
	On the other hand suppose $u \in G\setminus H$. Then we may write $u = \sigma^m$ where $m = qd + r$ for integers $q$ and $r$ with $0 < r < d$. Then $\chi^{n/d}(u) = \zeta_n^{qn + rn/d} = \zeta_d^r$. Hence
	\begin{eqnarray}
		\chi_\psi(u) & = & \sum_{k=0}^{d-1} \chi_G^{kn/d}(u) \nonumber\\
		& = & \sum_{k=0}^{d-1} \zeta_d^{kr} = \frac{\zeta_d^{dr} - 1}{\zeta_d^{r} - 1} = 0. \nonumber
	\end{eqnarray}
	Thus,
	\begin{equation*}
		\chi_\psi(u) = \begin{cases*}
			0 & if $u \notin H$, \\
			d & if $u \in H$.
		\end{cases*}
	\end{equation*}
	For the character of $\rho$ we have
	\begin{eqnarray}
		\chi_\rho(u) & = & \sum_{t \in G/H, \hspace{0.1 cm} t^{-1}ut \in H} \mathbbm{1}_H(t^{-1}ut) \nonumber\\
		& = & \sum_{t \in G/H, \hspace{0.1 cm} u \in H} \mathbbm{1}_H(u) \nonumber \\
		& = & \begin{cases*}
			0 & if $u \notin H$,\\
			d & if $u \in H$
		\end{cases*} \nonumber
	\end{eqnarray}
	where the first equality follows from \cite[Thm.~12, Ch.~3]{Ser}. So $\chi_\psi = \chi_\rho$.
\end{proof}

\

Let $\varphi: R_K(G) = \Z[\chi_G] \to \Z[X]/\l X^n - 1\r$ denote the second isomorphism from Equation \ref{cyclring}. Then by the above proposition,
\begin{equation*}
	\varphi\left(\Ind_H^G(\mathbbm{1}_H)\right) = \varphi\left(\bigoplus_{k = 0}^{d-1}\chi_G^{kn/d}\right) = \sum_{k=0}^{d-1}X^{kn/d}.
\end{equation*}
Furthermore, if $P_{n,d} \in \Z[X]$ denotes the polynomial $\sum_{k=0}^{d-1}X^{kn/d}$, then
\begin{equation*}
	\varphi\left(\sum_{H < G} \Ind_H^G(R_K(H))\right) = \left\langle P_{n,d}: d > 1, d\mid n\right\rangle
\end{equation*}
as ideals in $\Z[X]$. Thus we may restate Theorem \ref{mainthm} as follows.

\begin{thm}\label{reform}
	Let $n$ be a positive integer. Then
	\begin{equation}
		\left\langle P_{n,d}: d > 1, d\mid n\right\rangle = \left\langle \Phi_n(X)\right\rangle
	\end{equation}
	as ideals in $\Z[X]$.
\end{thm}

\section{Preliminary Results} \label{sec4}

In this section we establish a few properties of the cyclotomic polynomials which will help us in proving Theorem \ref{reform}.

\begin{lem}\label{prep1}
Let $p$ be a prime and $r$ a positive integer such that $\gcd(r,p) = 1$.
\begin{itemize}
\item For any integer $m \geq 1$, 
\begin{equation*}
\P_{rp^m}(X) = \P_{rp}(X^{p^{m-1}}).
\end{equation*}
\item We have
\begin{equation*}
\P_{rp}(X) = \frac{\P_r(X^p)}{\P_r(X)}.
\end{equation*}
\end{itemize}
\end{lem}

\begin{proof}
	Since 
	\begin{equation*}
		X^n - 1 = \prod_{d\mid n} \Phi_d(X),
	\end{equation*}
	we obtain by the M\"{o}bius inversion formula
	\begin{equation}\label{mobius}
		\Phi_n(X) = \prod_{d\mid n}(X^d - 1)^{\mu(n/d)}
	\end{equation}
	where $\mu: \Z \to \{-1,0,1\}$ is the M\"{o}bius function defined as
	\begin{equation*}
		\mu(n) = \begin{cases*}
			0 & if $n$ is not squarefree, \\
			1 & if $n$ is squarefree and has an even number of prime factors, \\
			-1 & if $n$ is squarefree and has an odd number of prime factors.
		\end{cases*}
	\end{equation*}
	We now prove the first part of the lemma. From Equation \ref{mobius} we have
	\begin{equation}\label{4.1}
		\Phi_{rp^m}(X) = \prod_{d\mid rp^m}(X^d - 1)^{\mu(rp^m/d)}.
	\end{equation}
	However if $\ord_p(d) \leq m-2$ then $rp^m/d$ would be divisible by $p^2$ and consequently $\mu(rp^m/d) = 0$. Thus the product in Equation \ref{4.1} ranges over all divisors $d$ of $rp^m$ such that $\ord_p(d) = m$ or $m-1$. Writing $f = dp^{-(m-1)}$,
	\begin{eqnarray}
	 	\Phi_{rp^m}(X) & = & \prod_{f\mid rp}(X^{fp^{m-1}} - 1)^{\mu(rp^m/fp^{m-1})} \nonumber \\
	 	& = & \prod_{f\mid rp}(X^{fp^{m-1}} - 1)^{\mu(rp/f)} \nonumber \\
	 	& = & \Phi_{rp}(X^{p^{m-1}}). \nonumber
	\end{eqnarray}

	For the second part, note that any divisor of $pr$ is either a divisor of $r$ or $p$ times a divisor of $r$. Thus
	\begin{eqnarray}
		\Phi_{rp}(X) & = & \prod_{d\mid rp} (X^d-1)^{\mu(rp/d)} \nonumber \\
		& = & \prod_{d\mid r} (X^d-1)^{\mu(rp/d)}\cdot \prod_{d\mid r} (X^{pd}-1)^{\mu(rp/dp)}. \nonumber
	\end{eqnarray}
	Furthermore, since $\gcd(p,r) = 1$, we see that $\mu(rp/d) = \mu(p)\mu(r/d) = -\mu(r/d)$. Hence
	\begin{eqnarray}
		\Phi_{rp}(X) & = & \prod_{d\mid r} (X^d-1)^{\mu(rp/d)}\cdot \prod_{d\mid r} (X^{dp}-1)^{\mu(rp/dp)} \nonumber \\
		& = & \prod_{d\mid r} (X^d-1)^{-\mu(r/d)}\cdot \prod_{d\mid r} (X^{dp}-1)^{\mu(r/d)} \nonumber \\
		& = & \frac{\Phi_r(X^p)}{\Phi_r(X)}. \nonumber
	\end{eqnarray}
\end{proof}

\begin{prop}\label{prep2}
Let $n$ be a positive multiple of $p$ and let $a = \ord_p(n) \geq 1$. Then
\begin{equation*}
\P_n(X) = \frac{\P_{np^{-a}}(X^{p^a})}{\P_{np^{-a}}(X^{p^{a-1}})}.
\end{equation*}
\end{prop}

\begin{proof}
	From the first part of Lemma \ref{prep1},
	\begin{equation*}
	\Phi_n(X) = \Phi_{p^a(np^{-a})}(X) = \Phi_{p(np^{-a})}\left(X^{p^{a-1}}\right)
	\end{equation*}
	and from the second part of the same lemma,
	\begin{equation*}
	\Phi_{p(np^{-a})}\left(X^{p^{a-1}}\right) = \frac{\Phi_{np^{-a}}\left(X^{p^{a}}\right)}{\Phi_{np^{-a}}\left(X^{p^{a-1}}\right)}
	\end{equation*}
	and the result follows.
\end{proof}

\begin{prop}\label{prep3}
Let $n$ be a positive multiple of $p$ and $a = \ord_p(n) \geq 1$. Then for a prime divisor $q$ of $n$ other than $p$,
\begin{equation}
\frac{X^{n/p} - 1}{(X^{n/pq}-1)\cdot\P_{np^{-a}}(X^{p^{a-1}})} \in \Z[X].
\end{equation}
\end{prop}

\begin{proof}
	Let $n = p_1^{a_1}p_2^{a_2}\cdots p_k^{a_k}$ where $p_1, p_2, \dots, p_k$ are the distinct prime divisors of $n$ and $a_1, a_2, \dots, a_k$ are positive integers. Without loss of generality, we prove the lemma when $p = p_k$ and $q = p_{k-1}$. Writing $m = p_1p_2\cdots p_{k-1}$,
	\begin{equation}\label{eq4}
	\prod_{d\mid m} \Phi_d(X) = X^m-1,
	\end{equation}
	and since $p_1p_2\cdots p_{k-2}$ divides $m$, 
	\begin{eqnarray}
	\prod_{d\mid m} \Phi_d(X) & = & P(X)\cdot\left[\prod_{d\mid p_1p_2\cdots p_{k-2}} \Phi_d(X)\right] \cdot \Phi_m(X) \nonumber \\ & = & P(X)\cdot\left(X^{p_1p_2\cdots p_{k-2}}-1\right)\cdot\Phi_m(X) \nonumber
	\end{eqnarray}
	for some polynomial $P \in \Z[X]$. So from Equation \ref{eq4} we conclude that
	\begin{equation*}
	\frac{X^{p_1p_2\cdots p_{k-1}}-1}{\left(X^{p_1p_2\cdots p_{k-2}}-1\right)\cdot\Phi_{p_1p_2\cdots p_{k-1}}(X)} = P(X) \in \Z[X].
	\end{equation*}
	Substituting $X^{p_1^{a_1-1}p_2^{a_2-1}\cdots p_k^{a_k-1}}$ for $X$ in the above expression,
	\begin{equation}\label{eq5}
	\frac{X^{n/p_k}-1}{\left(X^{n/p_{k-1}p_k}-1\right)\cdot\Phi_{p_1p_2\cdots p_{k-1}}\left(X^{p_1^{a_1-1}p_2^{a_2-1}\cdots p_k^{a_k-1}}\right)} \in \Z[X].
	\end{equation}
	By a repeated application of the first part of Lemma \ref{prep1},
	\begin{equation*}
    \Phi_{np_k^{-a_k}}\left(X^{p_k^{a_k-1}}\right) = \Phi_{p_1^{a_1}p_2^{a_2}\cdots p_{k-1}^{a_{k-1}}}\left(X^{p_k^{a_k-1}}\right) = \Phi_{p_1p_2\cdots p_{k-1}}\left(X^{p_1^{a_1-1}p_2^{a_2-1}\cdots p_k^{a_k-1}}\right)
	\end{equation*}
	and so substituting this into Equation \ref{eq5} we find
	\begin{equation*}
	\frac{X^{n/p_k}-1}{\left(X^{n/p_{k-1}p_k}-1\right)\cdot\Phi_{np_k^{-a_k}}\left(X^{p_k^{a_k-1}}\right)} \in \Z[X]
	\end{equation*}
	which proves the lemma.
\end{proof}


\begin{lem}\label{prep4}
For positive integers $a$ and $b$ with $d = \gcd(a,b)$,
\begin{equation*}
\l X^d - 1\r = \l X^a-1, X^b-1\r
\end{equation*}
as ideals in $\Z[X]$.
\end{lem}

\begin{proof}
	The inclusion $\l X^a-1, X^b-1\r \subseteq \l X^d - 1\r$ is clear. We may assume without loss of generality that $a \geq b$. Writing $r_0 = a$, $r_1 = b$, we obtain from the Euclidean algorithm
	\begin{eqnarray}
	r_0 & = & q_1r_1 + r_2 \nonumber \\
	r_1 & = & q_2r_2 + r_3 \nonumber \\
	&\vdots& \nonumber \\
	r_{\ell-1} & = & q_{\ell} r_{\ell} \nonumber
	\end{eqnarray}
	where $r_{\ell} = d$, and $0 \leq r_{i+1} < r_i$ for all $i = 1, 2, \dots, \ell$. Then for $i = 1,\dots,\ell-1$ the polynomial
	\begin{equation*}
	f_i(X) = \frac{X^{q_ir_i}-1}{X^{r_i}-1}\cdot X^{r_{i+1}}
	\end{equation*}
	has integer coefficients, and
	\begin{equation*}
	\left(X^{r_{i-1}}-1\right) - f_i(X)\left(X^{r_i}-1\right) = X^{r_{i+1}}-1.
	\end{equation*}
	Therefore $\left\langle X^{r_{i}}-1, X^{r_{i+1}}-1 \right\rangle \subseteq \left\langle X^{r_{i-1}}-1, X^{r_i}-1 \right\rangle$ for all $0 \leq i < \ell$, and it follows that
	\begin{eqnarray}
	\left\langle X^d-1 \right\rangle = \left\langle X^{r_\ell}-1 \right\rangle & \subseteq & \left\langle X^{r_{\ell-1}}-1, X^{r_{\ell}}-1 \right\rangle \nonumber \\
& \subseteq & \left\langle X^{r_{0}}-1, X^{r_{1}}-1 \right\rangle \nonumber \\ 
& = & \left\langle X^{a}-1, X^{b}-1 \right\rangle \nonumber
	\end{eqnarray}
	which completes the proof of the lemma.
\end{proof}

\section{Proof of Theorem \ref{reform}} \label{sec5}

First let us show that $\left\langle P_{n,d}: d > 1, d\mid n\right\rangle \subseteq \langle \Phi_n(X) \rangle$. We note that if $\zeta_n$ is a primitive $n^\mathrm{th}$ root of unity, then $\zeta_n^{n/d}$ is a primitive $d^\mathrm{th}$ root of unity and thus
\begin{equation*}
	P_{n,d}(\zeta_n) = \sum_{k=0}^{d-1}\zeta_n^{kn/d} = 0.
\end{equation*}
So $P_{n,d} = Q_d\cdot\Phi_n$ for some polynomial $Q_d \in \mathbb{Q}[X]$. But since both $P_{n,d}$ and $\Phi_n$ are monic polynomials in $\Z[X]$, it follows that $Q_d \in \Z[X]$ and therefore $Q_d \in \Z[X]$ for all $d > 1$, $d \mid n$. This shows the inclusion $\left\langle P_{n,d}: d > 1, d\mid n\right\rangle \subseteq \langle \Phi_n(X) \rangle$.

\

The following proposition establishes the inclusion in the other direction.

\begin{prop}
Let $n$ be a positive integer with prime divisors $p_1, p_2, \dots, p_t$. Then there exist polynomials $f_1,f_2,\dots,f_t \in \Z[X]$ such that
\begin{equation*}
\sum_{i=1}^t\frac{X^n-1}{X^{n/p_i}-1}\cdot f_i(X) = \P_n(X).
\end{equation*}
\end{prop}

\begin{proof}
	We use induction on the number of prime divisors $t$ of $n$. The case $t = 1$ simply rephrases Lemma \ref{prep1}. Indeed, if $n = p^m$ for a prime $p$ and a positive integer $m$ then
	\begin{equation*}
		\frac{X^n-1}{X^{n/p} - 1} = \frac{X^{p^m}-1}{X^{p^{m-1}}-1} = \P_p(X^{p^{m-1}}) = \P_{p^m}(X) = \P_n(X)
	\end{equation*}
	where the second equality follows from Lemma \ref{prep1}.

	So assume that the proposition holds for all positive integers having at most $k-1$ distinct prime divisors, and let $n = p_1^{a_1}p_2^{a_2}\cdots p_k^{a_k}$. Then the positive integer $m = np_k^{-a_k}$ has at most $k-1$ distinct prime factors, so by our hypothesis
	\begin{equation*}
		\sum_{i=1}^{k-1} \frac{X^m-1}{X^{m/p_i}-1}\cdot f_i(X) = \Phi_m(X)
	\end{equation*}
	for some polynomials $f_1, f_2, \dots, f_{k-1} \in \Z[X]$. Substituting $X^{p_k^{a_k}}$ for $X$,
	\begin{equation*}
		\sum_{i=1}^{k-1} \frac{X^{mp_k^{a_k}}-1}{X^{mp_k^{a_k}/p_i}-1}\cdot f_i\left(X^{p_k^{a_k}}\right) = \Phi_m\left(X^{p_k^{a_k}}\right)
	\end{equation*}
	or
	\begin{equation}\label{eq1}
		\sum_{i=1}^{k-1} \frac{X^n-1}{X^{n/p_i}-1}\cdot F_i(X) = \Phi_{np_k^{-a_k}}\left(X^{p_k^{a_k}}\right)
	\end{equation}
	where $F_i(X) = f_i\left(X^{p_k^{a_k}}\right) \in \Z[X]$. Now for $1 \leq i \leq k-1$, define the polynomials $g_i$ by
	\begin{equation*}
		g_i(X) = \frac{X^{n/p_k}-1}{\left(X^{n/p_ip_k}-1\right)\Phi_{np_k^{-a_k}}\left(X^{p_k^{a_k-1}}\right)}\cdot F_i(X).
	\end{equation*}
	Then it follows by Proposition \ref{prep3} that $g_i \in \Z[X]$ for all $i$. So by dividing Equation \ref{eq1} by $\Phi_{np_k^{-a_k}}\left(X^{p_k^{a_k-1}}\right)$,
	\begin{equation}\label{eq2}
		\sum_{i=1}^{k-1} \frac{\left(X^n-1\right)\left(X^{n/p_ip_k}-1\right)}{\left(X^{n/p_i}-1\right)\left(X^{n/p_k}-1\right)}\cdot g_i(X) = \frac{\Phi_{np_k^{-a_k}}\left(X^{p_k^{a_k}}\right)}{\Phi_{np_k^{-a_k}}\left(X^{p_k^{a_k-1}}\right)}.
	\end{equation}
	
	\
	
	For each $1 \leq i \leq k-1$ there exist polynomials $a_i, b_i \in \Z[X]$ by Lemma \ref{prep4} such that
	\begin{equation*}
		(X^{n/p_i}-1)a_i(X) + (X^{n/p_k}-1)b_i(X) = (X^{n/p_ip_k}-1).
	\end{equation*}
	Substituting for $(X^{n/p_ip_k}-1)$ in Equation \ref{eq2} and using Proposition \ref{prep2} on the right hand side of the equation,
	\begin{equation*}
		\sum_{i=1}^{k-1} (X^n-1)\cdot\left(\frac{a_i(X)}{X^{n/p_k}-1} + \frac{b_i(X)}{X^{n/p_i}-1}\right)\cdot g_i(X) = \Phi_n(X).
	\end{equation*}
	Upon rearranging the sum we obtain
	\begin{equation*}
		\sum_{i=1}^{k} \frac{X^n-1}{X^{n/p_i}-1}\cdot h_i(X) = \Phi_n(X)
	\end{equation*}
	where for $1 \leq i \leq k-1$,
	\begin{equation*}
		h_i(X) = b_i(X)g_i(X) \in \Z[X]
	\end{equation*}
	and
	\begin{equation*}
		h_k(X) = \sum_{i=1}^{k-1}a_i(X)g_i(X) \in \Z[X].
	\end{equation*}
	This completes the proof of the proposition, and therefore also of Theorem \ref{mainthm}.
\end{proof}


\begin{thebibliography}{}
	\bibitem[Ser]{Ser} J.-P. Serre.
	\textit{Linear Representations of Finite Groups},
	Springer-Verlag, New York Inc., 1977.
\end{thebibliography}
\end{document}